\documentclass[12pt,reqno]{amsart}

\usepackage{fixltx2e} 
\usepackage[utf8]{inputenc} 
\usepackage{amssymb, latexsym, stmaryrd, amsthm, dsfont, amsfonts, amsbsy, amsmath, mathrsfs} 
\usepackage{mathtools} 
\usepackage{bm, bbm} 
\usepackage{enumerate} 
\usepackage{verbatim} 
\usepackage{url}   
\usepackage{lscape} 
\usepackage{centernot}
\usepackage[dvipsnames]{xcolor}
\usepackage[colorinlistoftodos]{todonotes} 
\usepackage{nicefrac} 
\usetikzlibrary{calc}

\usepackage{todonotes}

\usepackage{microtype,tikz,tikz-cd}  
\makeatletter 
\def\MT@register@subst@font{\MT@exp@one@n\MT@in@clist\font@name\MT@font@list
 \ifMT@inlist@\else\xdef\MT@font@list{\MT@font@list\font@name,}\fi}
\makeatother

\makeatletter 
\newcommand{\myitem}[1]{%
\item[(#1)]\protected@edef\@currentlabel{#1}%
}
\makeatother

\usepackage[pdftex,bookmarks,bookmarksnumbered,linktocpage,   
         colorlinks,linkcolor=blue,citecolor=blue]{hyperref}
         
\usepackage[capitalize,noabbrev]{cleveref} 

\newcommand{\bit}{\begin{itemize}}    
\newcommand{\eit}{\end{itemize}}
\newcommand{\ben}{\begin{enumerate}}
\newcommand{\een}{\end{enumerate}}

\newcommand{\benroman}{\ben[\normalfont (i)]}  
\let\eroman\een

\newcommand{\bde}{\begin{description}}
\newcommand{\ede}{\end{description}}
\let\oper=\mathbb                               

\newcommand{\III}{\oper{I}}                     
\newcommand{\SSS}{\oper{S}}                     

\theoremstyle{theorem}
\newtheorem{Theorem}{Theorem}
\newtheorem{Theorem-n}{Theorem}

\newtheorem{Proposition}[Theorem]{Proposition}
\newtheorem{Modal Sahlqvist Theorem}[Theorem]{Modal Sahlqvist Theorem}
\newtheorem{Intuitionistic Sahlqvist Theorem}[Theorem]{Intuitionistic  Sahlqvist Theorem}
\newtheorem{Esakia Duality}[Theorem]{Esakia Duality}

\newtheorem{Main Lemma}[Theorem]{Main Lemma}
\newtheorem{Compactness Theorem}[Theorem]{Compactness Theorem}
\newtheorem{Los Theorem}[Theorem]{\LL o\'s' Theorem}
\newtheorem{Isbell Theorem}[Theorem]{Isbell's Zigzag Theorem}
\newtheorem{Diagram Lemma}[Theorem]{Diagram Lemma}

\newtheorem{Transfer Lemma}[Theorem]{Transfer Lemma}
\newtheorem{Subdirect Decomposition Theorem}[Theorem]{Subdirect Decomposition Theorem}

\newtheorem{Claim}[Theorem]{Claim}

\theoremstyle{definition}

\theoremstyle{remark}

\usepackage[mathscr]{euscript}
 \let\mathscr\relax 
\usepackage[scr]{rsfso}

\renewcommand{\int}{\mathsf{int}\,}

\bmdefine{\A}{A} 
\bmdefine{\C}{C}                                
                              
\bmdefine{\2}{2}
\bmdefine{\B}{B}
\bmdefine{\D}{D}
\bmdefine{\E}{E}

\bmdefine{\Term}{T} 
\bmdefine{\Free}{F}
\bmdefine{\Fb}{F}

\usepackage{scalerel}

\newcommand{\V}{\mathsf{V}}

\newcommand{\K}{\mathsf{K}}

\newcommand{\HHH}{\mathbb{H}}
\newcommand{\PPP}{\mathbb{P}}

\newcommand{\PPU}{\mathbb{P}_{\!\textsc{\textup{u}}}^{}}

\let\LL\L 
\renewcommand{\L}{\mathscr{L}}

\newcommand{\Con}{\mathsf{Con}}

\DeclareUnicodeCharacter{001B}{}

\begin{document}

\title{Congruence permutability in quasivarieties}

\author{Luca Carai, Miriam Kurtzhals, and Tommaso Moraschini}

\address{Luca Carai: Dipartimento di Matematica ``Federigo Enriques'', Universit\`a degli Studi di Milano, via Cesare Saldini 50, 20133 Milano, Italy}\email{luca.carai.uni@gmail.com}

\address{Miriam Kurtzhals and Tommaso Moraschini: Departament de Filosofia, Facultat de Filosofia, Universitat de Barcelona (UB), Carrer Montalegre, $6$, $08001$ Barcelona, Spain}
\email{mkurtzku7@alumnes.ub.edu and tommaso.moraschini@ub.edu}

\maketitle

\begin{abstract}
It is shown that a natural notion of congruence permutability for quasivarieties already implies ``being a variety''. The result follows immediately from \cite{CBCCong} and the sole aim of this note is to state it explicitly, together with a telegraphic proof.
\end{abstract}

We denote the class operators of closure under isomorphic copies, subalgebras, homomorphic images, direct products, and ultraproducts by $\III, \SSS, \HHH, \PPP$, and $\PPU$, respectively. A class of algebras is said to be:
\benroman
\item a \emph{variety} when it is closed under $\HHH, \SSS$, and $\PPP$;
\item a \emph{quasivariety} when it is closed under $\III, \SSS, \PPP$, and $\PPU$.
\eroman
While every variety is a quasivariety, the converse is not true in general.
We call \emph{proper} the quasivarieties that are not varieties. Examples of a proper quasivariety include the class of cancellative commutative monoids.

As quasivarieties need not be closed under $\HHH$, the following concept is often useful.\ Let $\mathsf{K}$ be a quasivariety and $\A$ an algebra (not necessarily in $\K$). A congruence $\theta$ of $\A$ is said to be a $\K$\emph{-congruence} when $\A / \theta \in \K$. When ordered under the inclusion relation, the set of $\K$-congruences of $\A$ forms an algebraic lattice $\Con_\K(\A)$ in which  meets are intersections (see, e.g., \cite[Prop.~1.4.7 \& Cor.~1.4.11]{Go98a}). 
When $\K$ is the quasivariety  of all algebras of a given type, $\mathsf{Con}_\K(\A)$ coincides with the congruence lattice $\mathsf{Con}(\A)$ of $\A$. Given a quasivariety $\K$ and an algebra $\A$, we denote by $\land, \lor$ and $\land^{\mathsf{K}}, \lor^{\mathsf{K}}$ the meet and join operations in $\mathsf{Con}(\A)$  and $\mathsf{Con}_{\mathsf{K}}(\A)$ respectively. Moreover, we will rely on the following observation, which is an immediate consequence of 
\cite[Lem.~4.2]{BR99}:
for every $a, b \in A$ and $X \subseteq \mathsf{Con}_\K(\A)$,
\begin{equation}\label{Eq : Cg is finitary for princiapl}
    \langle a, b \rangle \in \bigvee^\K X \iff \text{there exists a finite $Y \subseteq X$ such that }\langle a, b \rangle \in \bigvee^\K Y.
    \end{equation}

The following is 
 a straightforward corollary 
of the Homomorphism Theorem (see, e.g., \cite[Thm.~II.6.12]{BuSa00}).

\begin{Proposition}
A quasivariety $\K$ is a variety if and only if $\mathsf{Con}(\A) = \mathsf{Con}_\K(\A)$ for every $\A \in \K$.
\end{Proposition}

Given two binary relations $R_1$ and $R_2$ on a set $A$, we let
\[
R_1 \circ R_2 = \{ \langle a, b \rangle \in A \times A : \text{there exists }c\in A \text{ s.t. }\langle a, c \rangle \in R_1 \text{ and }\langle c, b \rangle \in R_2 \}.
\]
A variety $\mathsf{K}$ is said to be \emph{congruence permutable} when for every $\A \in \mathsf{K}$ and $\theta, \phi \in \mathsf{Con}(\A)$ we have $\theta \circ \phi = \phi \circ \theta$. Equivalently, $\K$ is congruence permutable if and only if $\theta \lor \phi = \theta \circ \phi$ for every $\theta,\phi \in \mathsf{Con}(\A)$ (see, e.g., \cite[Thm.~5.9]{BuSa00}). We will prove that the analogous version of the latter property for quasivarieties implies ``being a variety''. This observation is a direct consequence of the results in \cite{CBCCong}, as shown in the next proof.

\begin{Theorem}\label{Theorem : CP implies variety : appendix}
Let $\mathsf{K}$ be a quasivariety such that $\theta \lor^{\mathsf{K}} \phi = \theta \circ \phi$ for every $\A \in \mathsf{K}$ and $\theta, \phi \in \mathsf{Con}_{\mathsf{K}}(\A)$. Then $\mathsf{K}$ is a variety.
\end{Theorem}

    \begin{proof}
        In view of  \cite{CBCCong}, it suffices to show that $\mathsf{Con}_{\mathsf{K}}(\A)$ is a complete sublattice of $\mathsf{Con}(\A)$ for every $\A \in \mathsf{K}$. To this end, consider $\A \in \K$. As meets are intersections in both $\mathsf{Con}_{\mathsf{K}}(\A)$ and $\mathsf{Con}(\A)$ and $\mathsf{Con}_{\mathsf{K}}(\A) \subseteq \mathsf{Con}(\A)$, it suffices to show that 
        \begin{equation}\label{Eq : the useful inclusion : CP}
            \bigvee^\K X \subseteq \bigvee X \text{ for every }X \subseteq \mathsf{Con}_\K(\A).
        \end{equation}

        We begin with the following observation.
        \begin{Claim}\label{Claim : the unique claim of the short note}
For every $\theta, \phi \in \mathsf{Con}_\K(\A)$ we have $\theta \lor^\K \phi \subseteq  \theta \lor \phi$.
        \end{Claim}

\begin{proof}[Proof of the Claim]
Since $\theta \lor \phi$ contains $\theta \circ \phi$ (see, e.g., \cite[Thms.~I.4.7 \& II.5.3]{BuSa00}) and the assumptions ensure that $\theta \circ \phi = \theta \lor^\K \phi$, we conclude that $\theta \lor^\K \phi \subseteq  \theta \lor \phi$.
\end{proof}

To prove (\ref{Eq : the useful inclusion : CP}), consider $\langle a, b \rangle \in \bigvee^\K X$. By (\ref{Eq : Cg is finitary for princiapl}) there exists a finite $Y \subseteq X$ such that
\begin{equation}\label{Eq : the useful inclusion : CP : 2}
    \langle a, b \rangle \in \bigvee^\K Y.
\end{equation}
First, suppose that $Y$ is nonempty. By applying in sequence (\ref{Eq : the useful inclusion : CP : 2}), Claim \ref{Claim : the unique claim of the short note}, and $Y \subseteq X$, we obtain
\[
\langle a, b \rangle \in \bigvee^\K Y \subseteq \bigvee Y \subseteq \bigvee X
\]
and we are done.
Next, suppose that $Y$ is empty. In this case, $\bigvee^\K Y$ is the minimum of $\mathsf{Con}_\K(\A)$. As $\A \in \K$ by assumption, this minimum is the identity relation on $A$. Together with (\ref{Eq : the useful inclusion : CP : 2}), this yields $a = b$. Consequently, $\langle a, b \rangle$ belongs to every congruence of $\A$ and, in particular, to $\bigvee X$.
\end{proof}

We remark that Theorem \ref{Theorem : CP implies variety : appendix} does not imply that every subquasivariety of a congruence permutable variety $\V$ is also a variety (counterexamples are well known and easy to find, e.g., in the case where $\V$ is the variety of all Heyting algebras).

\end{document}